\documentclass[12pt,oneside,reqno]{amsart}
\usepackage{geometry}                % See geometry.pdf to learn the layout options. There are lots.
\usepackage[parfill]{parskip}    % Activate to begin paragraphs with an empty line rather than an indent
\usepackage{graphicx}
\usepackage{amssymb}
\usepackage{tikz}
\usetikzlibrary{calc}
\usepackage{bm}
\usepackage{mathtools}
\usepackage{mathabx}

\usepackage{amsmath,amsthm,amscd,amssymb}
\usepackage{latexsym}
\usepackage[colorlinks,citecolor=red,pagebackref,hypertexnames=false]{hyperref}

\numberwithin{equation}{section}
\theoremstyle{plain}
\newtheorem{theorem}{Theorem}[section]
\newtheorem{lemma}[theorem]{Lemma}
\newtheorem{corollary}[theorem]{Corollary}

\theoremstyle{definition}
\newtheorem{definition}[theorem]{Definition}

\theoremstyle{remark}

\newtheorem{case[theorem]}{Case}

\newcommand{\R}{\mathbb{R}}

\newcommand{\E}{\mathbb{E}}

\newcommand{\ds}{\displaystyle}

\begin{document}

\title{On Erd\H os Chains in the plane} 
%    Information for first author

\author{J. Passant}

\date{\today}

%\email{iosevich@math.rochester.edu}
\email{jpassant@ur.rochester.edu}

\address{Department of Mathematics, University of Rochester, Rochester, NY 14627}

%\thanks{This work was partially supported by the NSA Grant H98230-15-1-0319}

\begin{abstract}Let $P$ be a finite point set in $\mathbb{R}^2$ with the set of distance $n$-chains defined as
$$ \Delta_n(P)=\{(|p_1-p_2|,|p_2-p_3|,\ldots,|p_n-p_{n+1}|):p_i \in P\}.$$
We show that for $2\leq n=O_{|P|}(1)$ we have
$$|\Delta_n(P)|\gtrsim \frac{|P|^{n}}{\log^{\frac{13}{2}(n-1)}|P|}.$$
Our argument uses the energy construction of Elekes and a general version of Rudnev's rich-line bound implicit in \cite{R19}, which allows one to iterate very efficiently on intersecting nested subsets of Guth-Katz lines.

\noindent Let $G$ is a simple connected graph on $m=O(1)$ vertices with $m\geq 2$. Define the graph-distance set $\Delta_G(P)$ as 
$$ \Delta_G(P) = \{ (|p_{i}-p_{j}|)_{\{i,j\}\in E(G)} : p_i,p_j \in P\}.$$

Combining with results of Guth and Katz \cite{GK15} and Rudnev \cite{R19} with the above, if $G$ has a Hamiltonian path we have
$$  |\Delta_G(P)| \gtrsim \frac{|P|^{m-1}}{\text{polylog}|P|}. $$
\end{abstract} 

\maketitle

%\tableofcontents

\section{Introduction}

\vskip.125in 

\vskip.125in 

\noindent Given a set $P$ in $\mathbb{R}^d$, we define the distance set of $P$ as
$$\Delta_d(P) = \{|x - y| : x, y \in P\} \subseteq \mathbb{R}^d.$$
The famous distance conjecture of Erd\H os \cite{E46} asked what is the minimal number of distinct distances determined by a finite point set $P$ in $\mathbb{R}^d$?
This was resolved in the plane by Guth and Katz \cite{GK15} building upon the work of Elekes and Sharir \cite{ES11}. This followed decades of work by, among others, Moser \cite{M52}, Chung \cite{C84}, Chung-Szemer\' edi-Trotter \cite{CST92}, Sz\'ekely \cite{S97}, Solymosi-T\' oth \cite{ST01}, Tardos \cite{T03}. See the book of Garibaldi, Iosevich and Senger \cite{GIS11} for a more complete introduction.

Recently progress has been made by Iosevich and the author \cite{IP18} and Rudnev \cite{R19} on a variant of Erd\H os' conjecture that asks about larger configurations. Suppose one has a graph $G$ with $k$ vertices. What is the minimum number of distinct-distance realisations when one takes the vertices of $G$ from a set of $n$ elements and considers distances only along edges of the graph? 

When the graph concerned is the complete graph on two vertices we see that this is exactly the distinct-distance problem of Erd\H os, when the graph is a triangle the question asks for distinct congruence classes of triangles. To give the precise formulation, for a finite point set $P$, we define the graphical-distance set of $P$ as

$$ \Delta_G(P) = \{ (|p_{i}-p_{j}|)_{\{i,j\}\in E(G)} : p_i,p_j \in P\}.$$

Then one asks for a lower bound on the size of $\Delta_G(P)$, as $|P|$ grows. Using the integer lattice as an upper bound one conjectures for a connected graph $G$ on $n+1$ vertices that for all $\varepsilon>0$ one can find a constant $C_\varepsilon$ such that

\begin{equation}\label{Eq: GraphDistConj}
	|\Delta_G(P)| \geq C_\varepsilon |P|^{n-\varepsilon}. 
\end{equation}

Configurations in the Euclidean setting were studied by Fürstenberg, Katznelson and Weiss \cite{FKW90} in the context of positive upper density. They expanded distance results in positive density sets due to Bourgain \cite{B86} and Falconer-Marstrand \cite{FM86} to show that one can find triangles. This result was then greatly expanded by Ziegler \cite{Z06} who showed one could find any simplex. Lyall and Magyar \cite{LM20} recently provided a sharp extension of the result of Bourgain. Bennett, Iosevich and Taylor \cite{BIT16} building on earlier work of Chan, \L aba and Pramanik \cite{CLP16} answered a version building on the Falconer conjecture (see \cite{F85}), showing that if one takes a set of sufficiently high Hausdorff dimension then if the graph is a chain of any length the graphical-distance set contains an open set. Using the improvements to the Falconer threshold in the plane due to Guth, Iosevich, Ou and Wang \cite{GIOW19}, Ou and Taylor \cite{OT20} recently improved the threshold for chains.

Similar results have also been obtained over finite fields in work of Iosevich and Parshall \cite{IP19} and Iosevich, Jardine and McDonald \cite{IJM21}.

In the Erd\H os setting J\'anos Pach asked how many similar triangles are defined by a set of points in the plane. Solymosi and Tardos \cite{ST07} found the tight bound that a point set $P$ determines at most $O(|P|^4\log(|P|))$ similar triangles pairs\footnote{Essentially the similar triangle energy in the terminology of this paper.} using bounds on $k$-rich complex transformations. One can quickly adapt this to bound the set of similar triangles by $\Omega(|P|^2/\log(|P|))$. This bound was reproved by Rudnev \cite{R12}, who also improved the bound on classes of congruent triangles to $\Omega(|P|^2)$ using the framework established by Elekes-Sharir-Guth-Katz \cite{ES11, GK15}.

%Similar triangles can have different scales and thus can correspond to many distance-trangles.

Iosevich and the author provided the first class of graphs for which (\ref{Eq: GraphDistConj}) holds. They established that if $G$ is a minimally infinitesimally rigid connected graph on $n+1$ vertices then $ |\Delta_G(P)| \gtrsim |P|^{n}$. Where the lack of logarithm in the bound is expected as rigid graphs contain many loops.

\begin{figure}[h]
	\centering
	\begin{minipage}{0.45\textwidth}
		\centering
		\begin{tikzpicture}
			\draw[thick] (1,2)--(3,4)--(6.5,4.5)--(5,2)--(1,2);
			\draw[thick] (3,4)--(5,2);
			%		\draw[thick] (3,4)--(5.5,3.75)--(5,2);
			%		\draw[thick] (3,4)--(6.5,4.5)--(5,2);
			\foreach \Point in {(1,2), (3,4),(6.5,4.5),(5,2)}{
				\node at \Point {\textbullet};
			}
		\end{tikzpicture}\\
		\caption{Rigid Graph}\label{Fig: RigidGraph}
	\end{minipage}\hfill
	\begin{minipage}{0.45\textwidth}
		\centering
		\begin{tikzpicture}
			\draw[thick] (1,2)--(3,4)--(6.5,4.5)--(5,2)--(1,2);
			%			\draw[thick] (3,4)--(5,2);
			%		\draw[thick] (3,4)--(5.5,3.75)--(5,2);
			%		\draw[thick] (3,4)--(6.5,4.5)--(5,2);
			\foreach \Point in {(1,2), (3,4),(6.5,4.5),(5,2)}{
				\node at \Point {\textbullet};
			}
		\end{tikzpicture}\\
		\caption{Non-Rigid Graph}\label{Fig: NonRigidGraph}
	\end{minipage}
\end{figure}

Iosevich and the author \cite{IP18} also note that (\ref{Eq: GraphDistConj}) quickly follows from the pinned Erd\H os conjecture. One can see this by noting that if $T_G$ is a spanning tree for $G$ then  $|\Delta_G(P)| \geq |\Delta_{T_G}(P)|$ and thus to prove (\ref{Eq: GraphDistConj}) in generality it suffices to prove the conjecture for trees. Using the pinned version of the Erd\H os distance result gives many rich pins and one can use these to construct a sufficient number of trees to verify (\ref{Eq: GraphDistConj}).

\begin{figure}[h]
	\centering
	\begin{tikzpicture}
		\draw[thick] (1,2)--(3,4)--(5,3);
		\foreach \Point in {(1,2), (3,4), (5,3)}{
			\node at \Point {\textbullet};
		}
	\end{tikzpicture}\\
	\caption{The 2-chain or hinge}\label{Fig: 2Chain}
	
\end{figure}
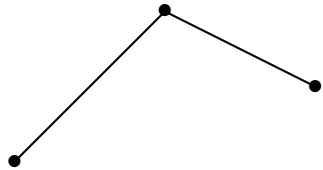

With this idea of trees in mind Iosevich and the author posed the question of whether one could verify (\ref{Eq: GraphDistConj}) for the 2-chain or hinge, the simplest non-rigid structure (see Figure \ref{Fig: 2Chain}). This was recently verified by Rudnev \cite{R19} who used a clever partitioning setup and a generalisation of the Guth-Katz incidence result on polynomial surfaces due to Sharir and Solomon. In this paper we extended Rudnev's result to all chains establishing that

\begin{equation}\label{Eq: k-Chain}
	|\Delta_{n-\text{chain}}(P)| \gtrsim \frac{|P|^{n}}{\log^{\frac{13}{2}(n-1)}|P|}.
\end{equation}

By the spanning-tree reduction we note that (\ref{Eq: k-Chain}) establishes (\ref{Eq: GraphDistConj}) for all chains of triangles and most generally any graph with a Hamiltonian path. We note also that (\ref{Eq: k-Chain}) doesn't apply to all rigid graphs, see Figure \ref{Fig: RigidNoPath}.

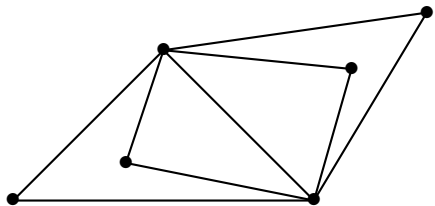
\begin{figure}[h]
	\centering
	\begin{tikzpicture}
		\draw[thick] (1,2)--(3,4)--(5,2)--(1,2);
		\draw[thick] (3,4)--(2.5,2.5)--(5,2);
		\draw[thick] (3,4)--(5.5,3.75)--(5,2);
		\draw[thick] (3,4)--(6.5,4.5)--(5,2);
		\foreach \Point in {(1,2), (3,4), (2.5,2.5), (5.5,3.75),(6.5,4.5),(5,2)}{
			\node at \Point {\textbullet};
		}
	\end{tikzpicture}\\
	\caption{A rigid graph with no Hamiltonian path}\label{Fig: RigidNoPath}
	
\end{figure}

Recent work has also been done on the number of realisations of a fixed chain--in the spirit of the unit distance conjecture--with first Palsson, Senger and Sheffer \cite{PSS19} providing strict upper bounds on the problem due to an example of Childs and lower bounds relating to the unit distance problem. These bounds were improved by Frankl and Kupavskii \cite{FK19} who removed the dependence on the unit distance conjecture in most cases. Neither of these results are strong enough to quickly establish a result as strong as (\ref{Eq: k-Chain}).

\subsection{Acknowledgements}
\noindent The author wishes to thank Alex Iosevich and Misha Rudnev for many helpful discussions and encouragement. The author wishes to thank Adam Sheffer, Josh Zahl and the participants of the MSRI summer school on the Polynomial method for many helpful discussions and MSRI, Berkeley for hosting the workshop. The author would also like to thank Doowon Koh for pointing out an error in a previous version of this paper.

\section{Statement of Results}

\noindent We prove the following

\begin{theorem}\label{Thm: GenChainBound}
	Let $P$ be a finite point set in $\mathbb{R}^2$ with the set of distance $n$-chains defined as
	$$ \Delta_n(P)=\{(|p_1-p_2|,|p_2-p_3|,\ldots,|p_n-p_{n+1}|):p_i \in P\}.$$
	We show that for $n=O_{|P|}(1)$ and $n\geq 3$ we have
	$$|\Delta_n(P)|\gtrsim \frac{|P|^{n}}{\log^{\frac{13}{2}(n-1)}|P|}.$$
\end{theorem}

When $n=1$ the above is the Erd\H os distinct distance problem resolved by Guth and Katz \cite{GK15}, with a dominator of $\log|P|$. When $n=2$ the result above was shown by Rudnev \cite{R19} with an improved denominator of $\log^3|P|$.
 
We note that Theorem \ref{Thm: GenChainBound} combined with these results resolves Conjecture (\ref{Eq: GraphDistConj}) for all Hamiltonian graphs on $O(1)$ vertices.

\begin{corollary}\label{Coro: GraphBound}
Let $P$ be a point set in $\mathbb{R}^2$.  Let $G$ is a connected simple graph with $m=O_{|P|}(1)$ vertices and $m\geq 2$. Then if $G$ contains a Hamiltonian path we have that
$$  |\Delta_G(P)| \gtrsim \frac{|P|^{m-1}}{\log^{\gamma(m)}|P|}, $$
where $\gamma(2)=1,\gamma(3)=3$ and $\gamma(m)=\frac{13}{2}(m-2)$ for $m\geq 4$.
\end{corollary}

We require a Hamiltonian path so that a spanning tree of $G$ is the $n$-chain. Corollary \ref{Coro: GraphBound} then follows immediately as the number of distance graphs of $G$ is bounded below by the number of distinct distance graphs of a spanning tree, see \cite[Section 4]{IP18}.

We note that Corollary \ref{Coro: GraphBound} is sharp up to $\log$ factors for any graph $G$. We expect that one can improve the $\log$ factor in the general case to $\log^{n-1}|P|$ and even further when the graph contains cycles. For example we expect further improvement in the case of the 4-cycle to

$$|\Delta_{C_4}(P)| \gtrsim \frac{|P|^{3}}{\log^2|P|},$$

an improvement here would be of extreme interest as we believe this a hard problem.

We note that the $l^2$ approach will fail to give a sharp bound for the 3-star. Indeed if one preforms the energy calculation below, see (\ref{Eq: CauchySchwarzEnergy}), then to obtain the sharp bound of $\Delta_{3-star}(P)\gtrsim |P|^{3-o(1)}$ one would need $\E_{3-star}(P)\lesssim |P|^5$. However, if one considers the example of a point set with $N$ points on each of three concentric circles and a point at their centre we have $\sim N^3$ realisations of the three star rooted at the centre. Choosing a pair of such centre rooted 3-stars gives a member of the energy thus $|\E_{3-star}(P)|\gtrsim |P|^6$.

%Using this we recover Rudnev's bound on 2-chains or hinges.
%
%\begin{corollary}\label{Coro: 2ChainSharp}
%	Let $P$ be a finite point set in $\mathbb{R}^2$ with the set of distance 3-chains defined as
%	$$ \Delta_2(P)=\{(|p_1-p_2|,|p_2-p_3|):p_i \in P\}.$$
%	We have that 
%	$$|\Delta_2(P)|\gtrsim \frac{|P|^{2}}{\log^3|P|}.$$
%\end{corollary}
%
%%By considering the $\sqrt{N} \times \sqrt{N}$ lattice we can see that Corollary \ref{Coro: 2ChainSharp} is the best one can do using the Elekes-Sharir-Guth-Katz framework, see \cite[Section 7]{S19HonThesis}.

\section{Proof of Theorem \ref{Thm: GenChainBound}}

To prove Theorem \ref{Thm: GenChainBound} we combine the famous breakthrough of Guth-Katz \cite{GK15} and the subsequent generalisations with a generalisation of a result of Rudnev from \cite{R19}.

First the generalisation of Guth-Katz's famous incidence result which one can find in \cite[Theorem 1.1]{SS18} and \cite[Theorem 12.1]{G16Book}

\begin{theorem}(Guth-Katz)\label{Thm: GuthKatz}
	Let $P$ be a set of points and $L$ be a set of lines in $\R^3$,  let $s$ be a parameter so that $|L|^{1/2}\leq s\leq |L|$ and no plane contains $s$ lines of $L$. The number of incidence between $P$ and $L$ then satisfies
	
	$$I(P,L) \lesssim |L|^{3/4}|P|^{1/2} + s^{1/3}|L|^{1/3}|P|^{2/3} + |L| + |P|.$$
\end{theorem}

We will use the following corollary,

\begin{corollary} \label{Coro: tRichPoints}
	Let $L$ be a set of lines in $\R^3$, let $s$ be a parameter so that $|L|^{1/2}\leq s$ and no plane contains $s$ lines of $L$. Let $P_t$ be the set of points where at least $t$ of these lines meet. Then there is a constant $t_0$ such that for $t\geq t_0$ we have 
	$$ |P_t| \lesssim \frac{|L|^{3/2}}{t^2} + \frac{s|L|}{t^3} + \frac{|L|}{t}. $$
\end{corollary}

In \cite{R19} Rudnev proves his main result by establishing the following theorem, though it is not explicitly stated\footnote{Rudnev shows that if one defines the $l^2$ energy $\E_2(P)=\{(p_1,p_2,p_3,p_1',p_2',p_3'): |p_i-p_{i+1}|=|p_i'-p_{i+1}'|\}$ related to the 2-chain one has $|\E_2(P)|\lesssim \log|L| \sum_{k,t}|L_{k,t}|k^2t^2$ and then establishes the bound $|L_{k,t}|k^2t^2\lesssim |L|^2$. Note that one gets two further $\log|L|$ terms from the support of the sums.}.

\begin{theorem}(Rudnev \cite{R19}) \label{Thm: RichLines}
	Suppose that $L$ is a set of lines so that we have no more than $|L|^{1/2}$ in any plane or regulus then if $L_{k,t}$ are the lines with $t$ points with $k$ lines of $L$ through them then we have
	
	$$|L_{k,t}| \lesssim \frac{|L|^2}{k^2t^2}.$$
\end{theorem}

We generalise this result to the following theorem which allow one to iterate on subsets.

\begin{theorem}\label{Thm: GenRichLines}
	Suppose that $L$ is a set of lines so that we have no more than $s\geq|L|^{1/2}$ in any plane or regulus and no more that $s\geq |L|^{1/2}$ lines concurrent. Then if $L_{k,t}$ are the lines of $L$ that contain $t$ points with $k$ lines of $L$ through them then we have
	
	$$|L_{k,t}| \lesssim \frac{|L|s^2}{k^2t^2} + \frac{|L|s\log(s)}{kt}.$$
\end{theorem}

We prove this in Appendix \ref{App: GenRichLines} as the proof only requires minor modifications to Rudnev's argument from \cite{R19}.
We will note the following corollary of Theorem \ref{Thm: GenRichLines}.

\begin{corollary}\label{Coro: kRichLines}
	If $L$ is a set of lines with no more than $s\geq|L|^{1/2}$ in any plane or regulus then if $L_r$ are the lines with at least $r$ lines of $L$ passing through them we have
	
	$$|L_r| \lesssim \frac{|L|s^2\log^2|L|}{r^2} + \frac{|L|s\log^3|L|}{r}.$$
\end{corollary}

\begin{proof}
We want to count $L_r$ we want to count the lines with $\sim t$ points with $\sim k$ lines through them such that $tk \gtrsim r$. So we estimate the following sum

$$ |L_r| = \sum_{tk\geq r} |L_{k,t}| \lesssim \sum_{tk\geq r} \left(\frac{|L|s^2}{k^2t^2} + \frac{|L|s\log(s)}{kt}\right), $$

using Theorem \ref{Thm: GenRichLines}. We then note that $\dfrac{1}{kt}\leq \dfrac{1}{r}$ and as the support of the sum is $\leq \log^2|L|$ we have the claimed result

$$ |L_r| \lesssim \frac{|L|s^2\log^2|L|}{r^2} + \frac{|L|s\log^3|L|}{r}. $$
\end{proof}

We note that Corollary \ref{Coro: kRichLines} is the best one can do as the Guth-Katz bound \linebreak $\sum_l \delta_{l,l'} \lesssim |L|^{3/2}\log|L|$ shows us that a typical line meets about $|L|^{1/2}\log|L|$ lines. So when $r=|L|^{1/2}\log|L|$ we have that $|L_r|\sim |L|$.

We note that the above doesn't follow from De Zeeuw's line-line incidence theorem \cite[Lemma 3.1]{dZ16} when $k=\Omega(1)$ as De Zeeuw's theorem requires one of the line sets to have $O(1)$ lines concurrent. 

It is Theorem \ref{Thm: GenRichLines} and Corollary \ref{Coro: kRichLines} which allow for the result here, as they allow us to continue to iterate the methodology of Theorem \ref{Thm: RichLines} to nested subsets of the Guth-Katz lines.

\subsection*{Proof of Theorem \ref{Thm: GenChainBound}}\label{Sec: Proof}~

We are now ready to prove the bound for all Erd\H os chains. We recall the setup of Elekes which allows us to count distances via counting the energy. Then we will use the ideas of Elekes-Sharir and Guth-Katz to count this energy via an incidence problem in $\R^3$.

First let $\nu(z) = |\{(p_1,p_2,\ldots, p_{n+1}) \in P^{n+1} : |p_i-p_{i+1}|=z_i \text{~~~for~~~} i=1,\ldots, n\}|$ be the number of times the $n$-chain with distances $z=(z_1,z_2,\ldots,z_n)$ arises. Then we can count the number of $n$-chains using

\begin{equation}\label{Eq: CauchySchwarzEnergy}
\left(|P|^{n+1}\right)^2 = \left(\sum_{z \in \Delta_n(P)} \nu(z)\right)^2 \leq |\Delta_n(P)|\sum_z \nu^2(z).
\end{equation}

We note that this final sum gives the size of the following energy set,

$$ \E_n(P) = \{(p_1,\ldots, p_{n+1}, p_1', \ldots, p_{n+1}'): |p_i-p_{i+1}| = |p_i' - p_{i+1}'| \text{~~~for~~~} i=1,\ldots, n\}. $$ 

So we aim to establish the bound $|\mathbb{E}_n(P)| \lesssim |L|^{\frac{n+2}{2}}\log^{\frac{7}{2}(n-1)}|L|$.

These energies can be thought of as configurations of Guth-Katz lines in $\R^3$ each entry in $\E_n$ corresponds to the $n+1$ lines $l_{p_1p_1'},\ldots,l_{p_{n+1}p_{n+1}'}$ having the intersections

$$ l_{p_ip_i'} \cap l_{p_{i+1}p_{i+1}'} \text{~~~for~~~} i=1,\ldots, n, $$

as in Figures \ref{Fig: 3Chain} and Figure \ref{Fig: 4Chain} for the 3-chain and 4-chain respectively.

%insert figure of 3 and four chain.

\begin{figure}[h]
	\centering
	\begin{minipage}{0.45\textwidth}
		\centering
			\begin{tikzpicture}
		\draw[thick] (1,2)--(3,4)--(6,2)--(7,3);
		\draw (1,6)--(3,10);
		\draw (2,10)--(5,6);
		\draw (3,6)--(6,10);
		\draw (4,10)--(7,6);
		\foreach \Point in {(1,2), (3,4), (6,2), (7,3)}{
			\node at \Point {\textbullet};
		}
		%		\foreach \point in {(1,1.5), (-1,2.5)}{
		%			\node at \point {\color{red} \textbullet};
		%		}
		\node[below] at (1,2) {$p_1$};%R^2 labels
		\node[above] at (3,4) {$p_2$};
		\node[below] at (6,2) {$p_3$};
		\node[above] at (7,3) {$p_4$};
		
		\node[below] at (1,6) {$l_{p_1p_1'}$};%R^3 labels
		\node[left] at (2,10) {$l_{p_2p_2'}$};
		\node[right] at (6,10) {$l_{p_3p_3'}$};
		\node[below] at (7,6) {$l_{p_4p_4'}$};
		
		\node at (4,7.33) {\textbullet};
		\node[left] at (4,7.33) {$p$};
		
		\node at (4,5) {$\updownarrow$};
		
		\end{tikzpicture}\\
		\caption{The 3-Chain}\label{Fig: 3Chain}
	\end{minipage}\hfill
	\begin{minipage}{0.45\textwidth}
		\centering
			\begin{tikzpicture}
		\draw[thick] (1,2)--(2,4)--(4,2)--(6,3)--(7,2);
		\draw (1,6)--(3,10);
		\draw (2,10)--(4,6);
		\draw[very thick] (3,6)--(5,10);
		\draw (4,10)--(6,6);
		\draw (5,6)--(7,10);
		\foreach \Point in {(1,2), (2,4), (4,2), (6,3), (7,2)}{
			\node at \Point {\textbullet};
		}
		%		\foreach \point in {(1,1.5), (-1,2.5)}{
		%			\node at \point {\color{red} \textbullet};
		%		}
		\node[below] at (1,2) {$p_1$};%R^2 labels
		\node[above] at (2,4) {$p_2$};
		\node[below] at (4,2) {$p_3$};
		\node[above] at (6,3) {$p_4$};
		\node[below] at (7,2) {$p_5$};
		
		\node[below] at (1,6) {$l_{p_1p_1'}$};%R^3 labels
		\node[left] at (2,10) {$l_{p_2p_2'}$};
		\node[right] at (5,10) {$l_{p_3p_3'}$};
		\node[right] at (6,6) {$l_{p_4p_4'}$};
		\node[below] at (5,6) {$l_{p_5p_5'}$};
		
		\node[left] at (4,8) {$l$};
		
		\node at (4,5) {$\updownarrow$};
		
		\end{tikzpicture}\\
		\caption{The 4-Chain}\label{Fig: 4Chain}
	\end{minipage}
\end{figure}

We separate our approach depending on whether $n$ is odd or even. For $n$ odd we have an even number of lines in the configuration in $\R^3$ and thus we have a central point, see $p$ in Figure \ref{Fig: 3Chain}. For $n$ even we have an odd number of lines and thus we have a central line, see $l$ in Figure \ref{Fig: 4Chain}.

We simplify notation using the following definitions.

\pagebreak

\begin{definition}\label{Defn: NestedPointLineSets}
	Let $s$ and $t$ positive integers. Let $L$ be a set of lines and $M\subseteq L$. Let $P$ be a subset of the points of intersection of $L$. Then we denote by
	\begin{itemize}
		\item $P_t(M)$ points of intersection of lines in $M$ which have between $t$ and $2t$ lines of $L$ thought them.
		\item $L_s(M)$ lines of $L$ that contain between $s$ and $2s$ lines of $M$ through them.
		\item We use $L_s$ to mean $L_s(L)$.
		\item $L_s(P_t(M))$ are lines of $L$ that contain between $s$ and $2s$ points of $P_t(M)$.
	\end{itemize}
\end{definition}

As our chains get longer the number of variables necessary to parametrise their configurations will grow, with each new variable indicating a step out from the centre of the configuration. So for the 3-chain have two lines either side of the central point we will need two variables $t_1$ and $t_2$ to count them, we label these in the point set $P_{t_1}(L_{t_2})$ where such points contain $t_1$ lines of $L_{t_2}$ contained in them. The 5-chain would have 3 variables $t_1,t_2,t_3$ and points in $P_{t_1}(L_{t_2}(L_{t_3}))$ i.e. points that contain $t_1$ lines of $L_{t_2}(L_{t_3})$ which are lines that contain $t_2$ intersections with lines that have $t_3$ lines of $L$ through them.

For the $n$-chain one can see that $P_{t_1}(L_{t_2}(L_{t_3}(\cdots L_{t_{(n-1)/2}}(L_{t_{(n+1)/2}})\cdots))))$ is the relevant quantity. For ease of notation we will simplify this to $P_{t_1;t_2,\ldots,t_{(n+1)/2}}$.

We have a similar issue with even chains. When we consider the 4-chain requires three variables $t_1,t_2$ and $t_3$ with each representing a line stepping out from the central line (including the central line itself). We can characterise our central lines as belonging to the set $L_{t_1}(P_{t_2}(L_{t_3}))$ indicating lines that have $t_1$ points where there are $t_2$ lines of $L_{t_3}$ passing through them. Again we simplify this to $L_{t_1,t_2;\ldots,t_{(n+2)/2}}$ for the $n$-chain.

We claim that we can bound the energy $\E_n(P)$ by

\begin{equation}\label{Eq: EnergyEstimate}
 |\mathbb{E}_n(P)| \leq \begin{cases*}
\ds\log^{(n-1)/2}|L|\sum_{t_1,t_2,\ldots,t_{(n+1)/2}}|P_{t_1;t_2,\ldots,t_{(n+1)/2}}|(t_1\cdots t_{(n+1)/2})^2 & \text{if } n \text{ is odd},\\
\ds\log^{n/2}|L|\sum_{t_1,t_2,\ldots,t_{(n+2)/2}}|L_{t_1,t_2;\ldots,t_{(n+2)/2}}|(t_1\cdots t_{(n+2)/2})^2 & \text{if } n \text{ is even}.
\end{cases*}
\end{equation}

We will demonstrate the initial iteration in the case $n=3,4$, the further cases follow from further iteration of the argument.

\subsection*{3-Chain Setup}

In the case of the 3-chain, define $\nu(l)=\sum_{l' \in L}\delta_{l,l'}$, then define $\nu(p)=\sum_{p \in l}\nu(l)$. Then $\nu(p)$ counts the number of ways to start at $p$ and then step two lines out (with the first line containing $p$).

Counting pairs of such for each $p$ will give us the number of configurations of four lines centred at $p$ seen in Figure \ref{Fig: 3Chain}. So we can see that

$$|\E_3(P)| \leq \sum_{p\in P} \nu^2(p).$$

We then estimate $\nu(p)$. We observe that we can divide the lines of $L$ into sets $L_{t_2}$ where each line of $L_{t_2}$ has between $t_2$ and $2t_2$ intersections with other lines of $L$. We note that there are at most $2\log|L|$ such sets. Then $\nu(p) = \sum_{t_2}|L_{t_2}(p)|t_2$, where $|L_{t_2}(p)|$ are the number of lines of $L_{t_2}$ which pass through $p$. Cauchy-Schwarz tells us that,

$$ \nu^2(p) \leq 2\log|L|\sum_{t_2}|L_{t_2}(p)|^2t_2^2. $$

Recall that $P_{t_1}(L_{t_2})$ are the points that have between $t_1$ and $2t_1$ lines of $L_{t_2}$ through them we have,

$$\sum_{p\in P} \nu^2(p) \lesssim 2\log|L| \sum_{t_1}\sum_{t_2}|P_{t_1}(L_{t_2})|t_1^2t_2^2.$$

This verifies Claim (\ref{Eq: EnergyEstimate}) for $n=3$. We note that the support of both sums is $\sim \log|L|$, so for the 3-chain it suffices to establish the bound $|P_{t_1}(L_{t_2})| \lesssim |L|^{5/2}\log^{12}|L|$.

\subsection*{4-Chain Setup}

In the case of the 4-chain we again define $\nu(l)=\sum_{l'\in L} \delta_{l,l'}$ and then iterate this to set up $\nu_2(l)=\sum_{l' \in L}\nu(l')\delta_{l,l'}$. So $\nu_2$ gives pairs $(l',l'')$ where $l'$ intersects $l$ and $l''$ intersects $l'$.

We can thus see that counting pairs of these $\nu_2$-pairs i.e. $((l_1',l_1''),(l_2',l_2''))$ with $l_1'$ and $l_2'$ intersecting a shared line $l$, counts the types of arrangements required to count 4-chains. See Figure \ref{Fig: 4Chain}. So we have that

$$|\E_4(P)| \leq \sum_{l\in L(P)}\nu_2^2(l).$$

We thus estimate $\nu_2(l)$. Let $L_{t_3}$ be the lines of $L(P)$ that intersect between $t_3$ and $2t_3$ lines of $L(P)$ and let $P_{t_2}(L_{t_3})(l)$ be points on $l$ which have between $t_2$ and $2t_2$ lines of $L_{t_3}$ through them. We can then see that $\nu_2(l) = \sum_{t_2,t_3}|P_{t_2}(L_{t_3})(l)|t_2t_3$. As the support of $t_2$ and $t_3$ are both $\sim \log|L|$ we have

$$ \nu_2(l)^2 \lesssim \log^2|L| \sum_{t_2,t_3} |P_{t_2}(L_{t_3})(l)|^2t_2^2t_3^2. $$

Partitioning for a third time so that we have lines in $L_{t_1}(P_{t_2}(L_{t_3}))$, which we recall from Definition \ref{Defn: NestedPointLineSets} contain between $t_1$ and $2t_1$ points from $P_{t_2}(L_{t_3})$ we have that

$$|\E_4(P)| \lesssim \log^2|L|\sum_{t_1,t_2,t_3}|L_{t_1}(P_{t_2}(L_{t_3}))|t_1^2t_2^2t_3^2,$$

which verifies Claim (\ref{Eq: EnergyEstimate}) for $n=4$. 

To show Claim (\ref{Eq: EnergyEstimate}) in full generality one defines $\nu_k(p)$ and $\nu_k(l)$ iteratively (with $\nu_1=\nu$ defined above) as $\nu_{k+1}(p)=\sum_{p \in l} \nu_k(l)$ and $\nu_{k+1}(l) = \sum_{l'\in L} \nu_k(l)\delta_{l,l'}$ respectively. One then and preforms the same analysis as above.

To bound $|P_{t_1;\ldots, t_n}|$ and $|L_{t_1,t_2;\ldots,t_n}|$ we rely on the following key lemma, an iteration of Corollary \ref{Coro: kRichLines}.

We wish to iterate on lines containing many intersections with the previous set, for this we introduce distinct notation.

\begin{definition}\label{Defn: IteratedRichLines}
	Let $\mathcal{L}_{t_n}$ to be lines of $L$ that contain $\geq t_n$ intersections with the other lines of $L$. Define $\mathcal{L}_{t_1,t_2,\ldots,t_n}$ iteratively as the lines of $L$ that contain at least $t_1$ lines of $\mathcal{L}_{t_2,\ldots,t_n}$.
\end{definition}

We note that $\mathcal{L}_{t_1,\ldots, t_n}$ differs from $L_{t_1,t_2;\ldots, t_n}$ with the former being lines that have many intersections with $\mathcal{L}_{t_1,\ldots, t_{n-1}}$ and $L_{t_1,t_2;\ldots, t_n}$ being lines that have $\sim t_1$ intersections with $\sim t_2$ lines from $\mathcal{L}_{t_3,\ldots, t_{n}}$, as using Definitions \ref{Defn: IteratedRichLines} and \ref{Defn: NestedPointLineSets} one can see
$$L_{t_3}(L_{t_4}(\cdots L_{t_{(n-1)/2}}(L_{t_{(n+1)/2}})\cdots))) \subseteq \mathcal{L}_{t_3,\ldots, t_{n}}.$$

We will use $\mathcal{L}_{t_1,t_2,\ldots,t_n}$ going forward as this simpler iteration is much easier to work with and the gain from Lemma \ref{Lem: IterativeLines} below is strong enough to mitigate any losses in this simplification.

\begin{lemma}\label{Lem: IterativeLines}
	Let $L$ be a line set in $\R^3$ with at most $|L|^{1/2}$ lines in any regulus or plane and at most $|L|^{1/2}$ lines concurrent. Define $\mathcal{L}_{t_n}$ to be lines of $L$ that contain $\geq t_n$ intersections with the other lines of $L$. Define $\mathcal{L}_{t_1,t_2,\ldots,t_n}$ iteratively as the lines of $L$ that contain at least $t_1$ intersections with lines of $\mathcal{L}_{t_2,\ldots,t_n}$.
	
	We partition the $t_i$ into two sets depending on its value relative to $|L|^{1/2}$. Let $\{t_{i_a}\}^j_{a=1}$ be the $t_i$ such that $t_{i_a}\geq |L|^{1/2}$ and $\{t_{i_b}\}_{b=1}^k$ be the $t_i$ such that $t_i < |L|^{1/2}$. Note that $j+k=n$. Then
	
	$$ |\mathcal{L}_{t_1,\ldots,t_n}| \lesssim \frac{|L||L|^{j/2}|L|^k\log^{2k+3j}|L|}{\prod_{a=1}^jt_{i_a}\prod_{b=1}^k t^2_{i_b}}. $$
\end{lemma}

%\begin{proof}
%	For $n=1$ this is exactly Corollary \ref{Coro: kRichLines} applied to the whole line set $L$. For the induction step we note that $L_{t_1,\ldots,t_n}\subseteq L$ and thus there are at most $|L|^{1/2}$ in any plane or regulus, so we use $s=|L|^{1/2}$ and the line set $\mathcal{L}_{t_2,\ldots,t_{n}}$ the result follows using Corollary \ref{Coro: kRichLines} as
%	
%	$$ |\mathcal{L}_{t_1,\ldots,t_n}| \lesssim \frac{|\mathcal{L}_{{t_2,\ldots,t_{n}}}||L|\log^{2}|L|}{t_1^2} + \frac{|\mathcal{L}_{{t_2,\ldots,t_{n}}}||L|^{1/2}\log^{3}|L|}{t_1}, $$
%	
%	to which we then apply the induction hypothesis to gain
%	
%	$$ |\mathcal{L}_{t_1,\ldots,t_n}| \lesssim  \frac{|L|^{n+1}\log^{2n}|L|}{t_1^2\cdots t_n^2} + \frac{|L|^{(n+1)/2}|L|\log^{3n-1}|L|}{t_1^2t_2\cdots t_n} + \frac{|L|^{n}|L|^{1/2}\log^{2n+1}|L|}{t_1t^2_2\cdots t^2_n} +  \frac{|L|^{(n+2)/2}\log^{3n}|L|}{t_1\cdots t_n} $$
%	
%	If the second term is larger than the last we see that $t_1\lesssim \frac{|L|^{1/2}}{\log|L|}$, a contradiction as $t_1\geq |L|^{1/2}$. If the third term is larger than the last we see that $t_2\cdots t_n \lesssim \frac{|L|^{(n-1)/2}}{\log|L|}$ also a contradiction as $t_i\geq |L|^{1/2}$, thus we can reduce to the first and the last terms.
%\end{proof}

\begin{proof}
	We recall that Corollary \ref{Coro: kRichLines} tells us that
	
	$$ |\mathcal{L}_{t_1,t_2,\ldots,t_n}| \lesssim \frac{|\mathcal{L}_{t_2,\ldots,t_n}||L|\log^2|L|}{t_1^2} + \frac{|\mathcal{L}_{t_2,\ldots,t_n}||L|^{1/2}\log^3|L|}{t_1}. $$
	
	We see that the first term dominates if $t_1\leq |L|^{1/2}$ and the second term dominates if $t_1\geq |L|^{1/2}$. Taking the dominant term and repeating the process leads to the stated inequality. Note that at the final step you use
	
	$$|\mathcal{L}_{t_n}| \lesssim \frac{|L|^2 \log^2|L|}{t_n^2} + \frac{|L|^{3/2}\log^3|L|}{t_n},$$
	
	which gives the factor of $|L|^1$ present in the final inequality.
\end{proof}

We can now prove our bound on $n$-chains. We will first deal with $n$ odd and then derive the $n$ even result via a simple application of Cauchy-Schwarz.

\subsection{Proof of Theorem \ref{Thm: GenChainBound} for $n$ odd}~

By (\ref{Eq: EnergyEstimate}) it suffices to establish the estimate

\begin{equation}\label{Eq: OddCaseBound}
 |P_{t_1;t_2,\ldots,t_{(n+1)/2}}|(t_1\cdots t_{(n+1)/2})^2 \lesssim |L|^{(n+2)/2}\log^{6(n-1)}|L|.
\end{equation}

To do this we will combine the original inequality of Guth and Katz with the iterative bound in Lemma \ref{Lem: IterativeLines}.

For the first bound we note that points in $P_{t_1;t_2,\ldots,t_{(n+1)/2}}$ are a subset of $P_{t_1}(L)$ and thus we can apply the Guth-Katz bound directly to gain

\begin{equation}\label{Eq: GuthKatzEstimate}
	|P_{t_1;t_2,\ldots,t_{(n+1)/2}}|(t_1\cdots t_{(n+1)/2})^2 \lesssim |L|^{3/2}(t_2\cdots t_{(n+1)/2})^2,
\end{equation}

removing the factor of $t_1^2$. Playing (\ref{Eq: GuthKatzEstimate}) off against other bounds will be key to this proof. We also note that if $t_i\lesssim |L|^{1/2}$ uniformly that this suffices as

$$ |P_{t_1;t_2,\ldots,t_{(n+1)/2}}|(t_1\cdots t_{(n+1)/2})^2 \lesssim |L|^{3/2}(|L|^{1/2})^{(n-1)} \leq |L|^{(n+2)/2}.$$

We begin this more involved case by recalling that points in $P_{t_1;t_2,\ldots,t_{(n+1)/2}}$ are a subset of those with $\sim t_1$ lines of $\mathcal{L}_{t_2,\ldots,t_{(n+1)/2}}$ through them. So we apply Corollary \ref{Coro: tRichPoints} with the line set $\mathcal{L}_{t_2,\ldots,t_{(n+1)/2}}$. As $\mathcal{L}_{t_2,\ldots,t_{(n+1)/2}}$ is a subset of $L(P)$ we see that there are at most $s=|L|^{1/2}$ lines in any plane or regulus, thus

\begin{equation}\label{Eq: OddChainGuthKatz}
 |P_{t_1;t_2,\ldots,t_{(n+1)/2}}| \lesssim \frac{|\mathcal{L}_{t_2,\ldots,t_{(n+1)/2}}|^{3/2}}{t_1^2} + \frac{|\mathcal{L}_{t_2,\ldots,t_{(n+1)/2}}||L|^{1/2}}{t_1^3} + \frac{|\mathcal{L}_{t_2,\ldots,t_{(n+1)/2}}|}{t_1}.
\end{equation}

We will consider each term dominating as a separate case.

\underline{Case 1:} The term $\dfrac{|\mathcal{L}_{t_2,\ldots,t_{(n+1)/2}}|^{3/2}}{t_1^2}$ dominates (\ref{Eq: OddChainGuthKatz}).

Then we have that 

$$ |P_{t_1;t_2,\ldots,t_{(n+1)/2}}|(t_1\cdots t_{(n+1)/2})^2 \lesssim |\mathcal{L}_{t_2,\ldots,t_{(n+1)/2}}|^{3/2}(t_2\cdots t_{(n+1)/2})^2, $$

Using Lemma \ref{Lem: IterativeLines} we have that

\begin{align*}
|P_{t_1;t_2,\ldots,t_{(n+1)/2}}|(t_1\cdots t_{(n+1)/2})^2 &\lesssim \left(\frac{|L||L|^{j/2}|L|^k\log^{2k+3j}|L|}{\prod_{a=1}^jt_{i_a}\prod_{b=1}^k t^2_{i_b}}\right)^{3/2}(t_2\cdots t_{(n+1)/2})^2\\
 &= \frac{|L|^{3/2}|L|^{3j/4}|L|^{3k/2}}{\prod_{b=1}^k t_{i_b}}\left(\prod_{a=1}^jt_{i_a}\right)^{1/2} \cdot \log^{3k+(9/2)j}|L|.
\end{align*}

We compare this to the Guth-Katz bound (\ref{Eq: GuthKatzEstimate}).

\begin{align*}
|P_{t_1;t_2,\ldots,t_{(n+1)/2}}|(t_1\cdots t_{(n+1)/2})^2 &\lesssim |L|^{3/2}(t_2\cdots t_{(n+1)/2})^2\\
&\lesssim |L|^{3/2}\left(\prod_{b=1}^k t_{i_b}\right)^{1/2}\left(|L|^{k/2}\right)^{3/2}\left(\prod_{a=1}^jt_{i_a}\right)^2,
\end{align*} 

using that by definition $t_{i_b}\lesssim |L|^{1/2}$ for all $b$ (see Lemma \ref{Lem: IterativeLines}). Combining these bounds gives as the worst case scenario that

\begin{align*}
\frac{|L|^{3/2}|L|^{3j/4}|L|^{3k/2}}{\prod_{b=1}^k t_{i_b}}\left(\prod_{a=1}^jt_{i_a}\right)^{1/2} \log^{3k+(9/2)j}|L|	&=|L|^{3/2}\left(\prod_{b=1}^k t_{i_b}\right)^{1/2}|L|^{3k/4}\left(\prod_{a=1}^jt_{i_a}\right)^2 \\
{|L|^{3j/4}|L|^{3k/4}} \log^{3k+(9/2)j}|L|	&=\left(\prod_{a=1}^jt_{i_a}\right)^{3/2}\left(\prod_{b=1}^k t_{i_b}\right)^{3/2}.
\end{align*}

Thus we have

$$ {|L|^{j/2}|L|^{k/2}} \log^{2k+3j}|L|	= (t_2\cdots t_{(n+1)/2}), $$

which gives

$$ |L|^{j+k}\log^{4k+6j}|L| = (t_2\cdots t_{(n+1)/2})^2. $$

Using this estimate in (\ref{Eq: GuthKatzEstimate}) gives

$$ |P_{t_1;t_2,\ldots,t_{(n+1)/2}}|(t_1\cdots t_{(n+1)/2})^2 \lesssim |L|^{3/2}{|L|^{j+k}} \log^{4k+6j}|L|. $$

We know that $j+k = \frac{(n-1)}{2}$, as $j$ and $k$ count the number of $t_i$ excluding $t_1$, thus we have that

$$ |P_{t_1;t_2,\ldots,t_{(n+1)/2}}|(t_1\cdots t_{(n+1)/2})^2 \lesssim |L|^{\frac{(n+2)}{2}}\log^{6(n-1)}|L|. $$

Which gives the result as in Case 1.

\underline{Case 2:} The term $\dfrac{|\mathcal{L}_{t_2,\ldots,t_{(n+1)/2}}||L|^{1/2}}{t_1^3}$ dominates (\ref{Eq: OddChainGuthKatz}).

We apply Lemma \ref{Lem: IterativeLines} and obtain

\begin{align*}
	|P_{t_1;t_2,\ldots,t_{(n+1)/2}}|(t_1\cdots t_{(n+1)/2})^2 &\lesssim \frac{|\mathcal{L}_{t_2,\ldots,t_{(n+1)/2}}||L|^{1/2}}{t_1}(t_2\cdots t_{(n+1)/2})^2\\
	&\lesssim \frac{|L|^{3/2}|L|^{j/2}|L|^k\log^{2k+3j}|L|}{t_1\prod_{a=1}^jt_{i_a}\prod_{b=1}^k t^2_{i_b}}\cdot (t_2\cdots t_{(n+1)/2})^2\\
	&\lesssim \frac{|L|^{3/2}|L|^{j/2}|L|^k\log^{2k+3j}|L|}{t_1}\prod_{a=1}^jt_{i_a}
\end{align*}

where the products over $a$ and $b$ range over $t_2 \ldots t_{(n+1)/2}$. We compare this estimate to the one from (\ref{Eq: GuthKatzEstimate}), which as above we have

$$ |P_{t_1;t_2,\ldots,t_{(n+1)/2}}|(t_1\cdots t_{(n+1)/2})^2 \lesssim |L|^{3/2}|L|^{k}\left(\prod_{a=1}^jt_{i_a}\right)^2. $$

In the worst case we thus have that

$$\frac{|L|^{3/2}|L|^{j/2}|L|^k\log^{2k+3j}|L|}{t_1}\prod_{a=1}^jt_{i_a} =|L|^{3/2}|L|^{k}\left(\prod_{a=1}^jt_{i_a}\right)^2, $$

which simplifies to give

$$ \prod_{a=1}^jt_{i_a} = \frac{|L|^{j/2}\log^{2k+3j}|L|}{t_1}.$$

Using this and that $1/t_1\lesssim 1$ in (\ref{Eq: GuthKatzEstimate}) gives

$$ |P_{t_1;t_2,\ldots,t_{(n+1)/2}}|(t_1\cdots t_{(n+1)/2})^2 \lesssim |L|^{3/2}|L|^{k+j}\log^{4k+6j}|L| \lesssim |L|^{\frac{(n+2)}{2}}\log^{6(n-1)}|L|.$$

Recall $j+k = \frac{(n-1)}{2}$ as the support of the products they count above is $t_2,\ldots,t_{\frac{(n+1)}{2}}$.

\underline{Case 3:} The term $\dfrac{|\mathcal{L}_{t_2,\ldots,t_{(n+1)/2}}|}{t_1}$ dominates (\ref{Eq: OddChainGuthKatz}).

We again apply Lemma \ref{Lem: IterativeLines} to $\mathcal{L}_{t_2,\ldots,t_{(n+1)/2}}$, this gives

\begin{align*}
|P_{t_1;t_2,\ldots,t_{(n+1)/2}}|(t_1\cdots t_{(n+1)/2})^2 &\lesssim |\mathcal{L}_{t_2,\ldots,t_{(n+1)/2}}|t_1(t_2\cdots t_{(n+1)/2})^2 \\
&\lesssim \left(\frac{|L||L|^{j/2}|L|^k\log^{2k+3j}|L|}{\prod_{a=1}^jt_{i_a}\prod_{b=1}^k t^2_{i_b}}\right)\cdot t_1(t_2\cdots t_{(n+1)/2})^2\\
&= |L||L|^{j/2}|L|^k\log^{2k+3j}(|L|)t_1\prod_{a=1}^jt_{i_a}\\
&\leq |L|^{3/2}|L|^{j/2}|L|^k\log^{2k+3j}(|L|)\prod_{a=1}^jt_{i_a}.
\end{align*}

where this final line uses that $t_1 \leq |L|^{1/2}$ as $\mathcal{L}_{t_2,\ldots,t_{(n+1)/2}}$ is a subset of the Guth-Katz lines $L$ where at most $|L|^{1/2}$ can pass through any point.

We compare this estimate to the one from (\ref{Eq: GuthKatzEstimate}), which we recall gives

\begin{align*}
	|P_{t_1;t_2,\ldots,t_{(n+1)/2}}|(t_1\cdots t_{(n+1)/2})^2 &\lesssim |L|^{3/2}(t_2\cdots t_{(n+1)/2})^2\\
	&\lesssim |L|^{3/2}|L|^{k}\left(\prod_{a=1}^jt_{i_a}\right)^2.
\end{align*} 

Thus similar to Case 2 in the worst case scenario one has 

$$ \prod_{a=1}^jt_{i_a} = |L|^{j/2}\log^{2k+3j}|L|.$$

Using this with estimate (\ref{Eq: GuthKatzEstimate}) then gives

$$ |P_{t_1;t_2,\ldots,t_{(n+1)/2}}|(t_1\cdots t_{(n+1)/2})^2 \lesssim |L|^{3/2}|L|^{k+j}\log^{4k+6j}|L| \lesssim |L|^{\frac{(n+2)}{2}}\log^{6(n-1)}|L|.$$

\subsection{Proof of Theorem \ref{Thm: GenChainBound} for $n$ even}~
%By (\ref{Eq: EnergyEstimate}) we see that it suffice to show the following bound holds
%
%$$ |L_{t_1,t_2,\ldots,t_{(n+2)/2}}|(t_1\cdots t_{(n+2)/2})^2 \lesssim |L|^{(n+2)/2}\log^{n-2}|L|. $$
%
%By definition $L_{t_1,t_2,\ldots,t_{(n+2)/2}}$ is the lines that have $\sim t_1$ points with $\sim t_2$ of the line set $\mathcal{L}_{t_3,\ldots,t_{(n+2)/2}}$ on them. With this in mind we use Theorem \ref{Thm: GenRichLines} with the line set $\mathcal{L}_{t_3,\ldots,t_{(n+2)/2}}$. As $\mathcal{L}_{t_3,\ldots,t_{(n+2)/2}}$ is a subset of the Guth-Katz lines $L(P)$ we can see that no more than $|L|^{1/2}$ lie in any plane or regulus and thus we can set $s=|L|^{1/2}$. So Theorem \ref{Thm: GenRichLines} gives
%
%$$ |L_{t_1,t_2,\ldots,t_{(n+2)/2}}| \lesssim \frac{|\mathcal{L}_{t_3,\ldots,t_{(n+2)/2}}||L|}{t_1^2t_2^2}. $$
%%
%Applying Lemma \ref{Lem: IterativeLines} to $\mathcal{L}_{t_3,\ldots,t_{(n+2)/2}}$ we have
%%
%$$ |L_{t_1,t_2,\ldots,t_{(n+2)/2}}| \lesssim \frac{|L|^{(n-2)/2}|L|\log^{(n-2)}|L|}{t_3^2\cdots t^2_{(n+2)/2}}. $$
%
%Combining the above we have
%
%$$  |L_{t_1,t_2,\ldots,t_{(n+2)/2}}|(t_1\cdots t_{(n+2)/2})^2 \lesssim |L|^{(n+2)/2}\log^{(n-2)}|L|. $$

%\section{Long Chains Give Short Chains}
%
We show that any non-trivial bound on the $n$-chain energy gives a non-trivial upper bound on the $k$-chain energy for $k\leq n$.
%
%To make this statement precise we define the $n$-chain energy $\E_n$ as
%
%$$ \E_n(P) = \{(p_1,\ldots, p_{n+1}, p_1', \ldots, p_{n+1}'): |p_i-p_{i+1}| = |p_i' - p_{i+1}'| \text{~~~for~~~} i=1,\ldots, n-1\}. $$ 
%
%Using the same Cauchy-Schwarz argument we deployed in (\ref{Eq: CauchySchwarzEnergy}) we can see that a bound of $|P|^{n+2}\text{polylog}|P|$ suffices to verify Conjecture \ref{Conj: ConfConj}. We expect the sharp bound to be $|P|^{n+2}\log^{n}|P|$ for chains, with the number of log factors decreasing with the number of cycles in the graph $G$.
%
%With the $n$-chain energy established we have the following reduction theorem.
%
\begin{lemma}\label{Lem: LongChainToShortChain}
	Suppose that $n\geq 4$ is even, then
	$$ |\E_n(P)|^2 \lesssim |\E_{n-1}(P)||\E_{n+1}(P)|. $$
	Suppose that $n\geq 3$ is odd, then
	$$ |\E_n(P)|^2 \lesssim |\E_{n-2}(P)||\E_{n+2}(P)|. $$
\end{lemma}

\begin{proof}We prove this in the even case, the odd case follows similarly.

We define $\omega_{k}(p)$ to be the number of ways to from a $k$-chain of lines starting with a line through $p$ (one can see this as the $k^{th}$ iteration of the $\nu_i(p)$ from the beginning of Section \ref{Sec: Proof}). Thus to form the $(n+1)$-chain of Guth-Katz lines we can break them up into finding pairs of chains of length $\frac{n-1}{2}$ and $\frac{n+1}{2}$ (with the central line shared, see Figure \ref{Fig: 4Chain}).

With this in mind we observe that
	
\begin{align*}
	|\E_n(P)| = \sum_{p\in P} \omega_{(n/2)}(p)\omega_{(n+2)}(p) &\leq \left(\sum_{p\in P} \omega^2_{n/2}\right)^{1/2}\left(\sum_{p\in P} \omega^2_{(n+2)/2}\right)^{1/2}\\
	 &\sim |\E_{n-1}(P)|^{1/2}|\E_{n+1}(P)|^{1/2},
\end{align*}
	
	where the last line notes that to find the $(n-1)$-energy you count chains of lines of length $n$ and thus as $n$ is even there is a central point (see Figure \ref{Fig: 3Chain}), so counting pairs of $n/2$ chains through this point will give the energy.
\end{proof}

We use Lemma \ref{Lem: LongChainToShortChain} to give us the odd chain bound. Indeed,

\begin{align*}
 |\E_n(P)| &\lesssim |\E_{n-1}(P)|^{1/2}|\E_{n+1}(P)|^{1/2} \\
 		&\lesssim \left(|L|^{(n+1)/2}\log^{\frac{13}{2}(n-2)}\right)^{1/2}	\left(|L|^{(n+3)/2}\log^{\frac{13n}{2}}\right)^{1/2}
 		&=|L|^{(n+2)/2}\log^{\frac{13}{2}(n-1)}|L|,
\end{align*}

which concludes the proof of Theorem \ref{Thm: GenChainBound}.

\bibliography{../../DiscreteGeometryReferences}{} %../DiscreteGeometryReferences

\bibliographystyle{plain}

\renewcommand{\thesection}{\Alph{section}}
\setcounter{section}{0}

\section{Proof of Theorem \ref{Thm: GenRichLines}}\label{App: GenRichLines}

This proof exactly follows the proof of Rudnev with one minor deviation that I will point out. We require Theorem \ref{Thm: GuthKatz} along with the following two results.

\begin{theorem}(De Zeeuw \cite{dZ16})\label{Thm: LineLine}
	Let $L, L'$ be two sets of lines in $\R^3$, with $|L'| \leq |L|$ and at most $r$ of lines from $L$ lying in a plane or regulus. If $P$ is a set of all points where two distinct lines $l,l'$ with $l\in L$ and $l'\in L$ meet, then
	
	$$|P| \lesssim |L|\sqrt{|L'|} + r|L'|.$$
\end{theorem}

\begin{theorem}(Sharir-Solomon \cite{SS18})\label{Thm: SharirSolomom}
Let $P$ be a set of points and $L$ a set of lines in $\R^3$, lying in a degree $D\geq 2$ polynomial surface, which does not contain linear factors. Suppose, at most $s\leq D$ lines\footnote{As Rudnev \cite{R19} points out, the restriction $s\leq D$ is merely a consequence of the Bézout theorem; certainly $s$ can be replaced by a large quantity.} are contained in any plane. The number $I(P,L)$ of incidences between $P$ and $L$ satisfies the bound
$$ I(P,L) \lesssim |P|^{1/2}|L|^{1/2}d^{1/2} + |P|^{2/3}d^{2/3}s^{1/3} + |P| + |L|. $$

\end{theorem}

\begin{proof}[Proof of Theorem \ref{Thm: GenRichLines}]
	
We aim to show that for all $k$ and $t$ we have

\begin{equation}\label{Eq: GenBound}
|L_{k, t}|k^2t^2 \lesssim |L|s^2 \text{~~~~~or~~~~~} |L_{k, t}|kt \lesssim |L|s\log(s).
\end{equation}

\underline{Case 1:} $k$ is $O(1)$.
	
When $k$ is $O(1)$ we want to show $|L_{k,t}|t^2\lesssim |L|s^2$ we use Theorem \ref{Thm: LineLine} which tells us that

\begin{equation}\label{Eq: Case1Bound}
 t|L_{k,t}| \lesssim |L|\sqrt{|L_{k,t}|} + s|L_{k,t}|, 
\end{equation}

where we use that we have at most $s$ lines of $L$ in any plane or regulus. If the first term dominates then we have that

$$ |L_{k,t}|\lesssim \frac{|L|^2}{t^2}, $$

thus $|L_{k, t}|k^2t^2 \lesssim |L|^2k^2 \leq |L|s^2$ as $k=O(1)$ and $|L|^{1/2} \leq s$.

If the second term in (\ref{Eq: Case1Bound}) dominates, then $t\lesssim s$ and thus $|L_{k, t}|k^2t^2 \lesssim |L|s^2$ as $k=O(1)$.
	
This concludes case 1 and we assume that $k=\Omega(1)$.	

\underline{Case 2:} $k$ is $\Omega(1)$.

% Before we proceed we consider how to apply Corollary \ref{Coro: tRichPoints} in this setting. We have
%
%$$ |P_k| \lesssim \frac{|L|^{3/2}}{k^2} + \frac{|L|s}{k^3} + \frac{|L|}{k} $$
%
%Suppose that $\frac{|L|^{3/2}}{k^2} \geq \frac{|L|s}{k^3}$ then $s\leq |L|^{1/2}k$ but $k=\Omega(1)$ and $|L|^{1/2} \leq s$ so this is a contradiction, so the first term never dominates for $k=\Omega(1)$. Thus we have
%
%\begin{lemma}\label{Lem: LargekRichPoints}
%	Let $L$ be a set of lines in $\R^3$ with no more that $s\geq |L|^{1/2}$ lines in any plane or regulus. let $P_k$ be the points with $k$ lines of $L$ through them then for $k=\Omega(1)$ we have
%	$$ |P_k| \lesssim \frac{|L|s}{k^3} + \frac{|L|}{k} $$
%\end{lemma}

We let $P_k$ be the points with $k$ lines of $L$ through them. Restrict $P_k$ to those points supported on $L_{k,t}$, we apply Theorem \ref{Thm: GuthKatz}, noting that as $L_{k,t}\subset L$ at most $s$ lines of $L_{k,t}$ can lie in a plane or regulus. This gives us

$$ t|L_{k,t}| \lesssim I(P_k,L_{k,t}) \lesssim |P_k|^{1/2}|L_{k,t}|^{3/4} + s^{1/3}|L_{k,t}|^{1/3}|P_{k,t}|^{1/3} + |P_k|. $$

We brake these into subcases depending on which term dominates.

\underline{Case 2(i):} $t|L_{k,t}| \lesssim |P_k|^{1/2}|L_{k,t}|^{3/4}$.

If  $t|L_{k,t}| \lesssim |P_k|^{1/2}|L_{k,t}|^{3/4}$ then $|L_{k,t}|\lesssim \dfrac{|P_k|^2}{t^4}$ and so,

So (\ref{Eq: GenBound}) becomes

$$ |L_{k,t}|t^2k^2 \lesssim \frac{|P_{k}|^2}{t^2}k^2. $$

To progress we use Corollary \ref{Coro: tRichPoints} which shows that

$$ |P_k| \lesssim \frac{|L|^{3/2}}{k^2} + \frac{|L|s}{k^3} + \frac{|L|}{k} $$

so we have

$$ |L_{k,t}|t^2k^2 \lesssim \frac{|L|^3}{k^2t^2} + \frac{|L|^2s^2}{k^4t^2} + \frac{|L|^2}{t^2} $$

We can assume $kt\geq |L|^{1/2}$ otherwise (\ref{Eq: GenBound}) follows instantly (using $|L|\leq s^2$), so we have that the first of these is bounded by $|L|^2\leq |L|s^2$; the second is bounded by $|L|s^2$ and the last by $|L|^2\leq |L|s^2$.

\underline{Case 2(ii):} $t|L_{k,t}| \lesssim |P_k|$.

Now we have that $|L_{k,t}|\lesssim \dfrac{|P_k|}{t}$ so using Corollary \ref{Coro: tRichPoints} we have that

$$ |L_{k,t}|t^2k^2 \lesssim |P_k|tk^2 \lesssim |L|^{3/2}t + \frac{|L|s}{k}t + |L|kt$$

The third term is not a problem as $kt$ gives the number of lines that cross a fixed line (by definition of $L_{k,t}$) and thus $kt\leq |L|$. If the second term in Corollary \ref{Coro: tRichPoints} dominates we content ourselves with the bound $|L_{k,t}|\lesssim \frac{|L|s}{kt}$. 

So we assume that $|P_k|\lesssim \frac{|L|^{3/2}}{k^2}$ and focus on the first term which is sufficiently controlled if $t \leq s$. So we assume $t \gtrsim s \geq |L|^{1/2}$.

In light of Theorem \ref{Thm: SharirSolomom} we proceed by putting our set $L_{k,t}$ in the zero set $Z$ of a polynomial of degree 

\begin{equation}\label{Eq: DegreeZ1}
 D \lesssim \min\left\{ \sqrt{\frac{|P_k|}{t}}, |L|^{1/2}\right\}. 
\end{equation}

Our restricted $P_k$ thus lies in $Z$. We partition $L$ into $L_0$ and $L_Z$ where the lines in $L_0$ do not lie in the surface $Z$ and the lines $L_Z$ do. As a line in $L_0$ cannot meet $Z$ at more than $D$ points we have

$$ k|P_k| \lesssim I(P_k,L) \lesssim D|L| + I(P_k, L_Z). $$

If $D|L|$ dominates this bound we have

$$ |P_k| \lesssim \frac{|L|D}{k} \lesssim \frac{|L||P_k|^{1/2}}{t^{1/2}k}, $$

using (\ref{Eq: DegreeZ1}). Thus we have that 

$$ t \lesssim \frac{|L|^2}{k^2|P_k|} \lesssim \frac{|L|^2}{k^2t|L_{k,t}|},$$

the latter inequity using $|L_{k,t}|\lesssim \frac{|P_k|}{t}$. Rearranging the above gives $|L_{k,t}|t^2k^2 \lesssim |L|^2 \leq |L|s^2$. So we assume that $|L|D$ does not dominate.

By combining $k|P_k|\lesssim I(P_k,L_Z)$ and $|L|D \leq I(P_k,L_Z)$ we have

$$ |P_k| + |L|D \lesssim I(P_k,L_Z)$$

Again with Theorem \ref{Thm: SharirSolomom} in mind we remove from $Z$ any linear factors. 
%In a linear factor we can have at most $s$ lines and there are at most $D$ linear factors and thus we have that the linear factors can contain $\lesssim s^2D + |P_k|$ incidences.
%
%If $s^2D$ dominates the bound then we have that
%
%$$|P_k|k \lesssim s^2D \lesssim \frac{s^2|P_k|^{1/2}}{t^{1/2}}.$$
%
%Using that in this particular case we have $|P_k| \gtrsim t|L_{k,t}|$ we can see that
%
%$$ t \lesssim \frac{s^4}{|P_k|k^2} \lesssim \frac{s^4}{|L_{k,t}|tk^2}.$$  
%
%So we have that $|L_{k,t}|t^2k^2 \lesssim s^4$.

We have to count linear factors differently to Rudnev \cite{R19}. For us to bound intersections in linear factors we first bound the number of lines that can occur in planes. We do this using following variant of Corollary \ref{Coro: tRichPoints} for planes which contain $r$ lines which one obtains through point-plane duality in $\R^3$. We note that if we have at most $s$ lines concurrent the dual has at most $s$ in any plane, one can also show a dual regulus can contain at most $s$ lines of $L$, thus if $\Pi_r$ are planes that contain at most $r$ lines of $L$ we have

$$ |\Pi_r| \lesssim \frac{|L|^{3/2}}{r^2} + \frac{|L|s}{r^3} + \frac{|L|}{r}$$

Each plane with $r$ lines in can contain at most $r^2$ incidences and so we aim to bound $\sum_{r=1}^s r^2|\Pi_r|$.

\begin{align*}
\sum_{r=1}^s r^2|\Pi_{=r}| = \sum_{r=1}^s r|\Pi_{r}| &\lesssim \sum_{r=1}^s \frac{|L|^{3/2}}{r} +   \frac{|L|s}{r^2} + |L| \\
	&= |L|^{3/2}\log(s) + |L|s\lesssim |L|s\log(s).
\end{align*}

Thus we have $|P_k|k\lesssim |L|s\log(s)$ and thus

$$ |L_{k,t}|kt \lesssim |P_k|k \lesssim |L|s\log(s), $$

so $|L_{k,t}|\lesssim |L|s\log(s)(kt)^{-1}$.

WE now assume that the incidence from the the linear factors of $Z$ do not dominate, so we apply Theorem \ref{Thm: SharirSolomom} to find the non-linear (NL) incidences between $L_Z$ and $|P_K|$. We have %assuming NL incidences dominate

\begin{equation}\label{Eq: Case2iiNearlyThere}
k|P_k| \lesssim |L|D + |P_k| \lesssim I_{NL}(P_k,L_Z) \lesssim |P_k|^{1/2}D^{1/2}|L|^{1/2} + |P_k|^{2/3}D^{2/3}s^{1/3} + |L|.
\end{equation}

We again divide this up into cases where each term dominates. If $|P_k|^{1/2}D^{1/2}|L|^{1/2}$ dominates then we have that

$$ |L|D \lesssim |P_k|^{1/2}D^{1/2}|L|^{1/2}, $$

so $D \lesssim \frac{|P_k|}{|L|}$ then we can see that

$$ k|P_k| \lesssim |P_k|^{1/2}D^{1/2}|L|^{1/2} \lesssim |P_k|, $$

a contradiction as $k=\Omega(1)$.

If $|L|$ dominates in (\ref{Eq: Case2iiNearlyThere}) then $|P_k| \lesssim \frac{|L|}{k}$ and thus

$$ |L_{k,t}|t^2k^2 \lesssim |L|tk \lesssim |L|^2\lesssim |L|s^2. $$ 

If $|P_k|^{2/3}D^{2/3}s^{1/3}$ dominates in  (\ref{Eq: Case2iiNearlyThere}) then we use the inequalities

$$ k|P_k|, |L|D \lesssim |P_k|^{2/3}D^{2/3}s^{1/3} $$

to obtain $\frac{|L|^2k}{s} \lesssim |P_k|$. We then use the refined version of Corollary \ref{Coro: tRichPoints} (we ruled out the other terms of Corollary \ref{Coro: tRichPoints} dominating in the first paragraph of Case 2(ii)) which tells us

\begin{equation*}
 |P_k| \lesssim \frac{|L|^{3/2}}{k^2}.
\end{equation*}

We see that for $\frac{|L|^{3/2}}{k^2}$ to dominate Corollary \ref{Coro: tRichPoints} we have $s \leq k|L|^{1/2}$. The above then gives

$$ \frac{|L|^2k}{s} \lesssim |P_k| \lesssim \frac{|L|^{3/2}}{k^2},$$ 

which gives that $|L|^{1/2}k^3 \lesssim s$ which when we combine with $s \leq k|L|^{1/2}$ gives a contradiction as $k=\Omega(1)$.

This concludes Case 2(ii).

\underline{Case 2(iii):} $t|L_{k,t}| \lesssim s^{1/3}|P_k|^{2/3}|L_{k,t}|^{1/3}$.

Applying Corollary \ref{Coro: tRichPoints} we have

$$ |L_{k,t}| \lesssim \frac{s^{1/2}|P_k|}{t^{3/2}} \lesssim \frac{|L|^{3/2}s^{1/2}}{t^{3/2}k^2} + \frac{|L|s^{3/2}}{t^{3/2}k^3} + \frac{|L|s^{1/2}}{t^{3/2}k}. $$

The last term is not a problem as $t\leq |L|$. For the second term we have

$$ |L_{k,t}|kt \lesssim \frac{|L|s^{3/2}}{(tk)^{1/2}k^{3/2}} \lesssim |L|s, $$

this last using that $kt\geq s$ and thus $\frac{1}{(kt)^{1/2}} \leq \frac{1}{s^{1/2}}$. So we have $|L_{k,t}|\lesssim \frac{|L|s}{kt}$ again.

The first term is only problems if $t\gtrsim s$ so we assume this and repeat the analysis of Case 2(ii), setting up a surface $Z$ of degree

$$D \lesssim \min\left\{\frac{s^{1/4}|P_k|^{1/2}}{t^{3/4}},|L|^{1/2}\right\}.$$

We again have $k|P_k| \lesssim |L|D + I(P_k,L_Z)$, if the first term dominates then

$$ k|P_k| \lesssim |L|D \lesssim \frac{|L|s^{1/4}|P_k|^{1/2}}{t^{3/2}k^2}, $$

it follows that $|P_k|\lesssim \frac{|L|^2s^{1/2}}{t^{3/2}k^2}$. We plug this into  $t|L_{k,t}| \lesssim s^{1/3}|P_k|^{2/3}|L_{k,t}|^{1/3}$ to obtain

\begin{align*}
t|L_{k,t}| &\lesssim \frac{s^{2/3}|L|^{4/3}}{tk^{4/3}}|L_{k,t}|^{1/3}\\
t^{3/2}|L_{k,t}| &\lesssim \frac{s|L|^2}{t^{3/2}k^2}\\
|L_{k,t}|t^2k^2 &\lesssim \frac{s|L|^2}{t} \lesssim \frac{|L|s^3}{t} \lesssim |L|s^2
\end{align*}

the final inequality following from $t \gtrsim s$.

We again have to deal with the linear factors of $Z$, but these lead to $|L_{k,t}|\lesssim |L|s\log(s)(kt)^{-1}$ as in Case 2(ii). Indeed, we have $k|P_k|\lesssim |L|s\log(s)$ and plugging this into $|L_{k,t}| \lesssim \frac{s^{1/2}|P_k|}{t^{3/2}}$ gives

$$ |L_{k,t}|kt \lesssim \frac{|L|s^{3/2}\log(s)}{t^{1/2}}\lesssim |L|s\log(s), $$

again as $t\gtrsim s$.

We then obtain (\ref{Eq: Case2iiNearlyThere}) again. The first two terms are controlled in exactly the same way as in Case 2(ii). The final term means that $k|P_k| \lesssim |L|$ plugging this into  $t|L_{k,t}| \lesssim s^{1/3}|P_k|^{2/3}|L_{k,t}|^{1/3}$ gives

$$ |L_{k,t}|k^2t^2 \lesssim \frac{|L|s^{1/2}}{t^{1/2}}(kt) \lesssim |L|^2 \lesssim |L|s^2, $$

using that $kt\leq |L|$ and $t\gtrsim s$.
\end{proof}

\end{document}